\documentclass[a4paper,10pt,reqno]{amsart}
\usepackage{amsfonts,amsmath,amssymb,bbm,amsthm,amsbsy}
\usepackage{graphicx}
\usepackage{geometry}
\usepackage{float}
\usepackage{verbatim}
\usepackage{setspace}
\usepackage{siunitx}
\usepackage{paralist}
\usepackage{enumitem}
\usepackage{xcolor}
\usepackage{lipsum}
\usepackage[english]{babel}
\usepackage{booktabs}
\usepackage{array}
\usepackage{subfig}
\usepackage{caption}
\usepackage{mathtools}
\usepackage{stmaryrd}
\usepackage{stackengine,scalerel}
\usepackage{hyperref}
\usepackage[utf8]{inputenc}

\newenvironment{frcseries}{\fontfamily{frc}\selectfont}{}

\makeatletter
\newtheorem*{rep@theorem}{\rep@title}
\newcommand{\newreptheorem}[2]{
\newenvironment{rep#1}[1]{
 \def\rep@title{#2 \ref{##1}}
 \begin{rep@theorem}}
 {\end{rep@theorem}}}
\makeatother

\newtheorem{thm}{Theorem}[section]
\newtheorem{lemma}[thm]{Lemma}
\newtheorem{prop}[thm]{Proposition}
\newtheorem{corr}[thm]{Corollary}

\newtheorem*{thm*}{Theorem}
\newtheorem*{lemma*}{Lemma}
\newtheorem*{prop*}{Proposition}
\newtheorem*{corr*}{Corrolary}
\newtheorem*{claim*}{Claim}

\newtheorem{innercustomthm}{Theorem}
\newenvironment{customthm}[1]
  {\renewcommand\theinnercustomthm{#1}\innercustomthm}
  {\endinnercustomthm}

\newtheorem{innercustomcorr}{Corollary}
\newenvironment{customcorr}[1]
  {\renewcommand\theinnercustomcorr{#1}\innercustomcorr}
  {\endinnercustomcorr}

\newtheorem{innercustomrmk}{Remark}
\newenvironment{customrmk}[1]
  {\renewcommand\theinnercustomrmk{#1}\innercustomrmk}
  {\endinnercustomrmk}

\theoremstyle{remark}
\newtheorem{rmk}[thm]{Remark}

\newtheorem*{rmk*}{Remark}
\newtheorem*{conj*}{Conjecture}
\newtheorem*{quest*}{Question}

\theoremstyle{definition}
\newtheorem{defn}[thm]{Definition}
\newtheorem{exmp}[thm]{Example}

\newtheorem*{defn*}{Definition}
\newtheorem*{exmp*}{Example}

\newreptheorem{theorem}{Theorem}
\newreptheorem{corollary}{Corollary}
\newreptheorem{proposition}{Proposition}

\usepackage{fancyhdr}
\usepackage{tikz}
\usepackage{tikz-cd}

\usepackage{IEEEtrantools}

\newenvironment{equ*}[1]{\begin{IEEEeqnarray*}{#1}}{\end{IEEEeqnarray*}}

\newcommand{\Pc}{\mathcal{P}}

\DeclareFontFamily{U}{mathx}{}
\DeclareFontShape{U}{mathx}{m}{n}{<-> mathx10}{}
\DeclareSymbolFont{mathx}{U}{mathx}{m}{n}
\DeclareMathAccent{\widecheck}{0}{mathx}{"71}

\newcommand{\inj}{\hookrightarrow}
\newcommand{\sur}{\twoheadrightarrow}

\DeclareMathOperator{\Hom}{Hom}

\DeclareMathOperator{\Fun}{\mathbf{Fun}}

\DeclareMathOperator{\Img}{Im}

\newcommand{\Set}{\mathbf{Set}}
\newcommand{\Grp}{\mathbf{Grp}}
\newcommand{\Ab}{\mathbf{Ab}}

\newcommand{\ModR}{\mathbf{Mod}(R)}

\newcommand{\Pro}{\mathbf{Pro}}

\newcommand{\Cat}{\mathbf{Cat}}

\newcommand{\PGrp}{\mathbf{PGrp}}

\newcommand{\PAb}{\mathbf{PAb}}

\newcommand{\PModR}{\mathbf{PMod}(R)}

\title{Profinite Cosheaves Valued in Pro-regular Categories}
\author{Jiacheng Tang}
\thanks{Email: jiacheng.tang@postgrad.manchester.ac.uk; Address: Alan Turing Building, University of Manchester, Manchester, M13 9PY, United Kingdom.}

\calclayout
\begin{document}
\maketitle

\begin{abstract}
We prove that the category of profinite cosheaves valued in a pro-regular category (satisfying mild assumptions) is itself canonically a pro-completion. As a corollary, we extend Wilkes's cosheaf-bundle equivalence from profinite modules to profinite groups.
\end{abstract}

\section{Introduction}
\label{sec1}

Profinite coproducts have been used extensively in the study of profinite groups and graphs (see \cite{haran}, \cite{melfree} and \cite{ribesgraph}). These are coproducts of families $\{A_x\}_{x\in X}$ of profinite groups or modules which are indexed not by a set, but by a profinite space $X$. There has been recent work that studies these coproducts categorically (see \cite{gareth}, \cite{boggi}, \cite{jc3} and \cite{jc4}). In \cite{gareth}, Wilkes shows that profinite coproducts of profinite modules can be equivalently described as global cosections objects of profinite cosheaves valued in the category of profinite modules. The current paper extends the theory from profinite modules to any pro-regular category satisfying mild additional assumptions.

\subsection*{Aim and Main Results} Our main result is the following categorical equivalence. Let $\mathbf{C}=\Pro(\mathbf{D})$, where $\mathbf{D}$ is a small regular category, and assume that $\mathbf{C}$ satisfies the assumptions before Example \ref{regeg}. Let $\mathbf{CoSh}(\mathbf{C})$ be the category of cosheaves valued in $\mathbf{C}$ (see Definition \ref{defcosheaf}). Briefly, $\mathbf{CoSh}(\mathbf{C})$ has objects pairs $(A,X)$, where $X$ is a profinite space and $A$ is a cosheaf over the topological space $X$ valued in $\mathbf{C}$. We also let $\mathbf{CoSh}(\mathbf{C})_\mathrm{fin}\inj\mathbf{CoSh}(\mathbf{C})$ be the full subcategory consisting of objects $(A,X)$, where $X$ is finite and every costalk $A_x\in\mathbf{C}$ of $A$ is actually in $\mathbf{D}$. Equivalently, $\mathbf{CoSh}(\mathbf{C})_\mathrm{fin}=\mathbf{Fam}(\mathbf{D})$ is the free finite coproduct completion of $\mathbf{D}$. Then, we have the following.

\begin{customthm}{\ref{maineq}}
With our assumptions, there is an equivalence $\mathbf{CoSh}(\mathbf{C})=\Pro(\mathbf{CoSh}(\mathbf{C})_\mathrm{fin})$.
\end{customthm}

The theorem makes certain exactness properties of $\mathbf{CoSh}(\mathbf{C})$ easy to verify.

\begin{customcorr}{\ref{exactcorr}}
The category $\mathbf{CoSh}(\mathbf{C})$ is extensive.
\end{customcorr}

Another consequence we have is that the profinite coproduct used in the literature is, in some sense, the only way of defining a coproduct indexed by profinite spaces which extends finite coproducts.

\begin{customcorr}{\ref{glil}}
The global cosections functor $\mathbf{CoSh}(\mathbf{C})\to\mathbf{C}\colon(A,X)\mapsto A(X)$ commutes with inverse limits. In particular, the global cosections functor $\mathbf{CoSh}(\mathbf{C})=\Pro(\mathbf{CoSh}(\mathbf{C})_\mathrm{fin})\to\mathbf{C}$ is the (unique) inverse-limit-preserving extension of the functor $$\mathbf{CoSh}(\mathbf{C})_\mathrm{fin}\to\mathbf{C}\colon (A,X)\mapsto A(X)=\coprod_{x\in X}A_x.$$
\end{customcorr}

Finally, by taking $\mathbf{D}$ to be the category of finite groups in the main theorem (so $\mathbf{C}=\PGrp$ is the category of profinite groups), we extend Wilkes's cosheaf-bundle equivalence (\cite[Theorem 3.5]{gareth}) from profinite modules to profinite groups. Let $\Pro$ denote the category of profinite spaces and fix $X\in\Pro$. Let $\mathbf{CoSh}(X,\PGrp)$ denote the category of cosheaves over $X$ valued in $\PGrp$. Also, let $\Grp(\Pro_{/X})$ denote the category of group objects in the slice category $\Pro_{/X}$.

\begin{customrmk}{\ref{corrrmk}}
There is an equivalence $\mathbf{CoSh}(X,\PGrp)=\Grp(\Pro_{/X})$.
\end{customrmk}

\subsection*{Outline of Paper}

We will assume basic knowledge of category theory (refer to \cite{maccat}), sheaf theory (refer to \cite{macsheaves}), pro-completions (refer to \cite{prolimits} or \cite{catandsh}) and regular categories (refer to \cite{barr}). Occasionally, we will use facts about extensive categories (refer to \cite{extensive}) and profinite groups (refer to \cite{profinite}).

In Section 2, we study the general theory of cosheaves over profinite spaces valued in a category $\mathbf{C}$ in which inverse limits commute with finite colimits. Most of this is well known when $\mathbf{C}=\Set^\mathrm{op}$, but we were unable to find a reference that includes our setting. In Section 3, we restrict our attention to pro-regular categories $\mathbf{C}=\Pro(\mathbf{D})$ with mild additional assumptions and prove our main results.

Remark: When we write ``=" in this paper, we usually mean canonically isomorphic (or equivalent in the case of categories).

Convention: Unless otherwise specified, all rings are associative with a 1 but not necessarily commutative. All our categories are locally small and all (co)limits are small.

\subsection*{Acknowledgements}
The author would like to thank his supervisor Peter Symonds for his constant guidance and Giacomo Tendas for helpful discussions and for reading drafts of this paper.

This work was supported by the EPSRC Doctoral Training Partnership [grant number \linebreak EP/W524347/1].

\section{Profinite Cosheaves}
\label{sec2}

In this section, we study categories of cosheaves over profinite spaces. Most of this is well known from (the dual of) usual sheaf theory when the base category is $\Set^\mathrm{op}$ or the opposite of any finitary algebraic theory over $\Set$ (such as $\Grp$). The book \cite{catandsh} studies sheaf theory valued in a category satisfying the so-called IPC-property, but our base category $\mathbf{C}$ doesn't have to satisfy the dual of the IPC-property. Nevertheless, because our cosheaf condition is particularly simple (see the next definition), it's easy (and useful) to give details below for the convenience of the reader. We recall that we write $\Pro=\Pro(\Set_\mathrm{fin})$ for the category of profinite spaces.

The important takeaway from this section is the key lemma (Lemma \ref{invsection}), which describes certain inverse limits in the category $\mathbf{CoSh}(\mathbf{C})$ of cosheaves valued in $\mathbf{C}$.

\begin{defn}\label{defcosheaf}
Let $\mathbf{C}$ be a category with finite coproducts, $X$ be a profinite space and $O_c(X)$ be the poset category of clopen subsets of $X$. A \emph{precosheaf over $X$ valued in $\mathbf{C}$} is a functor $A\colon O_c(X)\to \mathbf{C}$. A precosheaf $A$ over $X$ is a \emph{cosheaf} if it preserves the initial object and binary disjoint coproducts, i.e.\ if $A(\varnothing)=0$, and whenever $U,V\subseteq X$ are disjoint clopen, the natural map $A(U)\coprod A(V)\to A(U\sqcup V)$ is an isomorphism. We will write $\mathbf{PCoSh}(X,\mathbf{C})$ and $\mathbf{CoSh}(X,\mathbf{C})$ for the categories of precosheaves and cosheaves over $X$ respectively.

Let $\mathbf{CoSh}(\mathbf{C})$ be the Grothendieck construction of the functor $$\Pro\to\Cat, X\mapsto\mathbf{CoSh}(X,\mathbf{C}).$$ Spelt out, this means that $\mathbf{CoSh}(\mathbf{C})$ is a category with objects pairs $(A,X)$, where $X\in\Pro$ and $A\in\mathbf{CoSh}(X,\mathbf{C})$. We will often abuse notation and write $A$ for the pair $(A,X)$ when $X$ is understood. Given two cosheaves $(A,X)$ and $(B,Y)$, a morphism from $A$ to $B$ is a pair $(\varphi,f)$, where $f\colon X\to Y$ is a continuous map and $\varphi$ is a natural transformation $A\circ f^{-1}\to B$.
\[
\begin{tikzcd}
                                             & O_c(X) \arrow[rd, "A"] \arrow[d, "\varphi", Rightarrow] &         \\
O_c(Y) \arrow[ru, "f^{-1}"] \arrow[rr, "B"'] & {}                                       & \mathbf{C}
\end{tikzcd}
\]
\end{defn}

Let us first point out that the cosheaves defined above are equivalent to usual cosheaves, provided that $\mathbf{C}$ has all colimits. This is probably well known and a special case of this is stated on \cite[pages 4-5]{jc3} with a brief explanation. Let $\mathbf{CoSh}(X,\mathbf{C})''$ be the full subcategory of the functor category $\Fun(O(X),\mathbf{C})$ containing functors which satisfy the usual cosheaf condition, i.e.\ whenever $U=\bigcup_i U_i$ is an open cover of an open $U\subseteq X$, the following canonical diagram is a coequaliser: $$\coprod_{i,j}A(U_i\cap U_j)\rightrightarrows\coprod_iA(U_i)\to A(U).$$ Here, $O(X)$ is the poset of open subsets of $X$. Let $\mathbf{CoSh}(X,\mathbf{C})'$ be the full subcategory of $\mathbf{PCoSh}(X,\mathbf{C})=\Fun(O_c(X),\mathbf{C})$ containing functors which satisfy the cosheaf condition for arbitrary clopen covers $U=\bigcup_i U_i$. Then, we have equivalences $$\mathbf{CoSh}(X,\mathbf{C})=\mathbf{CoSh}(X,\mathbf{C})'=\mathbf{CoSh}(X,\mathbf{C})''.$$

In the following, we will assume that $\mathbf{C}$ has all limits and finite colimits, such that inverse limits commute with finite colimits. We now state several properties of $\mathbf{CoSh}(X,\mathbf{C})$.

\begin{prop}\label{coshcomplete}
The category $\mathbf{CoSh}(X,\mathbf{C})$ is a full subcategory of $\mathbf{PCoSh}(X,\mathbf{C})$ closed under all colimits and inverse limits (which exist in $\mathbf{PCoSh}(X,\mathbf{C})$).
\begin{proof}
The cosheaf condition is about finite coproducts in $\mathbf{C}$, which commute with colimits and inverse limits in $\mathbf{C}$.
\end{proof}
\end{prop}

\begin{prop}\label{cosheafification}
The inclusion $\mathbf{CoSh}(X,\mathbf{C})\inj\mathbf{PCoSh}(X,\mathbf{C})$ has a right adjoint, denoted by $(-)^{\mathrm{cosh}}$ and called \emph{cosheafification}.
\begin{proof}
This is analogous to the classical construction of sheafification. For a clopen subset $U\subseteq X$, let $\Pc(U)$ be the poset category of finite clopen partitions of $U$. That is, $\Pc(U)$ has objects partitions $U=\bigsqcup_{i=1}^nU_i$ with $U_i$ clopen, and $\{U_i\}\leq \{U_j\}$ if $\{U_i\}$ is a refinement of $\{U_j\}$. It's easy to see that $\Pc(U)$ is cofiltered, i.e.\ codirected.

Given a precosheaf $A$ and $\varnothing\neq U\subseteq X$ clopen, define $A^\mathrm{cosh}(U)=\varprojlim_{\{U_i\}\in\Pc(U)}\coprod_i A(U_i)$. We leave it to the reader to check that this makes $A^\mathrm{cosh}$ a precosheaf, i.e.\ a functor. To see that it is in fact a cosheaf, we simply note that if $U,V\subseteq X$ are disjoint clopen, then $\Pc(U)\times\Pc(V)\subseteq\Pc(U\sqcup V)$ is cofinal, and that inverse limits commute with finite coproducts in $\mathbf{C}$.

Finally, suppose $B$ is a cosheaf and we have a map $B\to A$. For $U\subseteq X$ clopen, we need to give a map $B(U)\to A^\mathrm{cosh}(U)$ in $\mathbf{C}$, which we take to be the map induced by the maps $B(U)=\coprod_iB(U_i)\to\coprod_iA(U_i)$ for partitions $\{U_i\}$ of $U$. One then verifies that this establishes $(-)^\mathrm{cosh}$ as being right adjoint to inclusion.
\end{proof}
\end{prop}

\begin{corr}\label{cosheaflimits}
The category $\mathbf{CoSh}(X,\mathbf{C})$ has all limits which are computed as cosheafifications of the precosheaf limits.
\end{corr}

\begin{rmk}
We note that Proposition \ref{cosheafification} is a special case of \cite[Theorem 3.1]{cosheafification}, but the construction of our cosheafification functor is particularly simple due to the specific Grothendieck site we are working with.
\end{rmk}

\begin{prop}
Cosheafification is exact.
\begin{proof}
It is certainly left exact, being a right adjoint. It also preserves finite colimits from the construction above and the fact that inverse limits are exact in $\mathbf{C}$.
\end{proof}
\end{prop}

Let $X\in\Pro$. There is a functor $\Delta\colon\mathbf{C}\to\mathbf{PCoSh}(X,\mathbf{C})$, the \emph{constant precosheaf functor}, which sends an object $c\in\mathbf{C}$ to the precosheaf $U\mapsto c$. Composing this with the cosheafification functor, we obtain a functor $\mathbf{C}\to\mathbf{CoSh}(X,\mathbf{C})$, still denoted by $\Delta$ and called the \emph{constant cosheaf functor}.

\begin{prop}\label{constantadj}
The constant cosheaf functor $\Delta\colon\mathbf{C}\to\mathbf{CoSh}(X,\mathbf{C})$ is right adjoint to the \emph{global cosections functor} $A\mapsto A(X)$.
\begin{proof}
We clearly have $$\Hom_{\mathbf{CoSh}(X,\mathbf{C})}(A,\Delta(c))=\Hom_{\mathbf{PCoSh}(X,\mathbf{C})}(A,\Delta(c))=\Hom_{\mathbf{C}}(A(X),c).$$
\end{proof}
\end{prop}

\begin{exmp}\label{egcosheaf}
\begin{enumerate}[label=(\roman*)]
\item\label{egcosheaf1} Suppose $X$ is finite. Then, a cosheaf $A\in\mathbf{CoSh}(X,\mathbf{C})$ is completely determined by the values $A_x=A(\{x\})$ for $x\in X$. The global cosections functor sends $A$ to the finite coproduct $A(X)=\coprod_{x\in X}A_x$.
\item\label{egcosheaf2} Suppose $X=I\sqcup\{\infty\}$ is the one-point compactification of a set $I$. Note that the clopen subsets of $X$ are precisely the finite subsets not containing $\infty$ and their complements. Suppose $\mathbf{C}$ also has all coproducts. Given a family of objects $\{c_i\}_{i\in I}$ in $\mathbf{C}$, we can define a cosheaf $C\in\mathbf{CoSh}(X,\mathbf{C})$ via $C(\{x\})=c_x$ if $x\neq\infty$ and $C(U)=\coprod_{j\in U\backslash\{\infty\}}c_j$ if $U$ is a cofinite subset containing $\infty$. The global cosections object of $C$ is the coproduct $C(X)=\coprod_{i\in I}c_i$.
\item\label{egcosheaf3} Suppose $X=I\sqcup\{\infty\}$ is still the one-point compactification of a set $I$, and that finite products and finite coproducts agree in $\mathbf{C}$ (e.g.\ if $\mathbf{C}$ is additive). Given a family of objects $\{c_i\}_{i\in I}$ in $\mathbf{C}$, we can define a cosheaf $P\in\mathbf{CoSh}(X,\mathbf{C})$ via $P(\{x\})=c_x$ if $x\neq\infty$ and $P(U)=\prod_{j\in U\backslash\{\infty\}}c_j$ if $U$ is a cofinite subset containing $\infty$. The global cosections object of $P$ is the product $P(X)=\prod_{i\in I}c_i$. Hence, with the right assumptions on $\mathbf{C}$, we see that global cosections functors of profinite cosheaves valued in $\mathbf{C}$ capture both products and coproducts in $\mathbf{C}$.
\end{enumerate}
\end{exmp}

Next, we shall study properties of the Grothendieck construction $\mathbf{CoSh}(\mathbf{C})$. The following observation is obvious.

\begin{prop}
The functor $$\mathbf{C}\to\mathbf{CoSh}(\mathbf{C}), c\mapsto (c,*)$$ is right adjoint to the global cosections functor $(A,X)\mapsto A(X)$.
\end{prop}

To study limits in $\mathbf{CoSh}(\mathbf{C})$, we need the following analogue of \cite[Theorem 2$^\prime$]{csgrothendieck}. Note that the following is not exactly the dual of \cite[Theorem 2$^\prime$]{csgrothendieck}, but its proof is sufficiently similar that we will omit it.

\begin{prop}\label{propgcomplete}
Let $\mathbf{I}$ be a small category. Suppose $\mathbf{E}$ is a category and $F\colon\mathbf{E}\to\Cat$ is a functor such that the following hold:
\begin{itemize}
\item The category $\mathbf{E}$ has $\mathbf{I}$-limits (i.e.\ limits of shape $\mathbf{I}$).
\item For each object $e\in\mathbf{E}$, the category $Fe$ has $\mathbf{I}$-limits.
\item For each morphism $f\colon e\to e'$ in $\mathbf{E}$, the functor $Ff\colon Fe\to Fe'$ has a right adjoint.
\end{itemize}
Then, the Grothendieck construction $\int F$ has $\mathbf{I}$-limits.
\end{prop}

Our main theorem (Theorem \ref{maineq}) will show that $\mathbf{CoSh}(\mathbf{C})$ is a pro-completion under additional assumptions, so $\mathbf{CoSh}(\mathbf{C})$ should have all inverse limits first.

\begin{prop}\label{coshRlim}
The category $\mathbf{CoSh}(\mathbf{C})$ has all limits.
\begin{proof}
Recall that $\mathbf{CoSh}(\mathbf{C})$ is defined as the Grothendieck construction of the functor $$F\colon\Pro\to\Cat, X\mapsto\mathbf{CoSh}(X,\mathbf{C}).$$ Thus, the first bullet point of Proposition \ref{propgcomplete} is obvious and so is the second point by Corollary \ref{cosheaflimits}. For the final point, note that given a map $f\colon X\to Y$ in $\Pro$, the functor $\mathbf{CoSh}(X,\mathbf{C})\to\mathbf{CoSh}(Y,\mathbf{C})$ is simply the \emph{direct image functor} $f_*\colon A\mapsto A\circ f^{-1}$. As in (the dual of) usual sheaf theory, it has a right adjoint $f^*$, the \emph{inverse image functor}, which sends $A\in \mathbf{CoSh}(Y,\mathbf{C})$ to the cosheaf $\left(U\mapsto\varprojlim_{V\supseteq f(U)}A(V)\right)^\mathrm{cosh}$, where $U\subseteq X$ is clopen.
\end{proof}
\end{prop}

It is useful to know explicitly how inverse limits in $\mathbf{CoSh}(\mathbf{C})$ are computed (cf.\ \cite[Theorem 6.5]{nunes}). Given an inverse system $\{(A_i,X_i)\}$, its inverse limit $(A,X)$ in $\mathbf{CoSh}(\mathbf{C})$ has $X=\varprojlim X_i$, taken in $\Pro$. Let $f_i\colon X\to X_i$ be the projection maps and $f_i^*\colon\mathbf{CoSh}(X_i,\mathbf{C})\to\mathbf{CoSh}(X,\mathbf{C})$ be the inverse image functors. Then, $\{f_i^*(A_i)\}$ naturally forms an inverse system in $\mathbf{CoSh}(X,\mathbf{C})$ and we let $A=\varprojlim f_i^*(A_i)$, taken in $\mathbf{CoSh}(X,\mathbf{C})$. We remind the reader that cosheafification is not needed to compute inverse limits here (see Proposition \ref{coshcomplete}).

To be more precise, given a map $(f\colon X_i\to X_j, \varphi\colon A_i\circ f^{-1}\to A_j)$ in the inverse system, the corresponding transition map $f_i^*(A_i)\to f_j^*(A_j)$ is the cosheafification of $$U\mapsto\left(\varprojlim_{V\supseteq f_i(U)}A_i(V)\overset{\psi}{\longrightarrow}\varprojlim_{W\supseteq f_j(U)}A_j(W)     \right),$$ where $\psi$ is induced by $$\varprojlim_{V\supseteq f_i(U)}A_i(V)\to A_i(f^{-1}(W))\overset{\varphi_W}{\longrightarrow} A_j(W).$$ This looks fairly complicated, but if the spaces $X_i$ are finite, then we have an easy description of the inverse limit $(A,X)$. In the following lemma (only), let's write $f_i^*$ to mean the precosheaf inverse image functor.

\begin{lemma}[Key Lemma]\label{invsection}
Suppose $\{(A_i,X_i)\}_{i\in I}$ is an inverse system in $\mathbf{CoSh}(\mathbf{C})$ with limit $(A,X)$ and where the $X_i$ are finite. Let $f_i\colon X\to X_i$ be the projection maps. Then, the precosheaf inverse limit $\varprojlim f_i^*(A_i)$ in $\mathbf{PCoSh}(X,\mathbf{C})$ is already a cosheaf. In particular, for $U\subseteq X$ clopen, we have $A(U)=\varprojlim A_i(f_i(U))$.
\begin{proof}
First, note that as each $X_i$ is discrete, we have $f_i^*A_i(U)=A_i(f_i(U))$ for each $U\subseteq X$ clopen.
Given $U,V\subseteq X$ disjoint, we can use properties of profinite spaces to find a single $j\in I$ such that $U=f_j^{-1}f_j(U)$ and $V=f_j^{-1}f_j(V)$ (cf.\ \cite[Exercise 1.1.15/Lemma 1.1.16]{profinite}). It's easy to check that if $i\geq j$, then $U=f_i^{-1}f_i(U)$ and $V=f_i^{-1}f_i(V)$. The important point is that $f_i(U)$ and $f_i(V)$ are disjoint. Thus, we have
\begin{eqnarray*}
\varprojlim_i f_i^*A_i(U\sqcup V)&=&\varprojlim_{i\geq j} A_i(f_i(U\sqcup V))\\
&=&\varprojlim_{i\geq j} A_i(f_i(U)\sqcup f_i(V))\\
&=&\varprojlim_{i} A_i(f_i(U))\textstyle{\,\coprod\,}\varprojlim_{i} A_i(f_i(V))\\
&=&\varprojlim_if^*_iA_i(U)\textstyle{\,\coprod\,}\varprojlim_if^*_iA_i(V).
\end{eqnarray*}
It only remains to note that cosheafification commutes with inverse limits.
\end{proof}
\end{lemma}

Given a cosheaf $A\in\mathbf{CoSh}(X,\mathbf{C})$ and $x\in X$, the \emph{costalk of $A$ at $x$} is $A_x=\varprojlim_{V\ni x}A(V)\in\mathbf{C}$. In other words, $A_x$ is the inverse image functor $i^*_x$, where $i_x\colon\{x\}\inj X$ is the inclusion map. The costalk functor $(-)_x\colon\mathbf{CoSh}(X,\mathbf{C})\to\mathbf{C}$ is, of course, right adjoint to the \emph{skyscraper cosheaf functor} $\mathrm{Sky}_x$, i.e.\ the direct image functor $(i_x)_*$. There are clearly also precosheaf versions of these.

We won't need the following two results, but they seem worth stating. Both can be proven by using Yoneda in a standard way.

\begin{prop}\label{laterprop}
Cosheafification preserves costalks. That is, for $A\in\mathbf{PCoSh}(X,\mathbf{C})$ and $x\in X$, we have $A^\mathrm{cosh}_x=A_x$.
\end{prop}

\begin{prop}\label{ptwise}
Let $(A,X)=\varprojlim(A_i,X_i)$ in $\mathbf{CoSh}(\mathbf{C})$ and $x=(x_i)\in X=\varprojlim X_i$. Then, we have $A_x=\varprojlim (A_i)_{x_i}$. That is, the costalks of an inverse limit in $\mathbf{CoSh}(\mathbf{C})$ are computed ``pointwise".
\end{prop}

\section{Main Theorem}
\label{sec3}

From now on, we shall assume that $\mathbf{C}=\Pro(\mathbf{D})$ is the pro-completion of a small regular category $\mathbf{D}$.

The main aim of this section is to prove our main theorem under suitable assumptions on $\mathbf{C}$. Let $\mathbf{CoSh}(\mathbf{C})_\mathrm{fin}$ be the full subcategory of $\mathbf{CoSh}(\mathbf{C})$ consisting of objects $(A,X)$, where $X$ is finite and the costalk $A_x\in\mathbf{D}$ for each $x\in X$. We should think of $\mathbf{CoSh}(\mathbf{C})_\mathrm{fin}$ as the subcategory of ``finite objects" of $\mathbf{CoSh}(\mathbf{C})$. Our goal is to prove the following.

\begin{customthm}{\ref{maineq}}
There is an equivalence $\mathbf{CoSh}(\mathbf{C})=\Pro(\mathbf{CoSh}(\mathbf{C})_\mathrm{fin})$.
\end{customthm}

Our strategy is the standard ``recognition principle".
\begin{itemize}
\item We show that every object of $\mathbf{CoSh}(\mathbf{C})$ is an inverse limit of objects in $\mathbf{CoSh}(\mathbf{C})_\mathrm{fin}$ (Lemma \ref{invdecomp}).
\item We show that for $\varprojlim A_i\in\mathbf{CoSh}(\mathbf{C})$ with $A_i\in\mathbf{CoSh}(\mathbf{C})_\mathrm{fin}$ and $B\in\mathbf{CoSh}(\mathbf{C})_\mathrm{fin}$, the canonical map $$\varinjlim\Hom_{\mathbf{CoSh}(\mathbf{C})}(A_i,B)\to\Hom_{\mathbf{CoSh}(\mathbf{C})}(\varprojlim A_i,B)$$ is surjective (Lemma \ref{factorone}).
\item Finally, we show that the above map is also injective (Theorem \ref{maineq}).
\end{itemize}

Here are some preliminary facts. We note that $\mathbf{C}$ is regular, has all limits and colimits, and that the embedding $\mathbf{D}\inj\mathbf{C}$ preserves colimits and finite limits. Moreover, we can view $\mathbf{C}=\Pro(\mathbf{D})$ as the dual of the full subcategory of left exact functors in $\Fun(\mathbf{D},\Set)$ and the embedding $\mathbf{C}=\Pro(\mathbf{D})\inj\Fun(\mathbf{D},\Set)^\mathrm{op}$ preserves colimits and inverse limits. It follows that inverse limits in $\mathbf{C}$ indeed commute with finite colimits, since direct limits in $\Set$ commute with finite limits, so the results of Section \ref{sec2} apply. In particular, inverse limits in $\mathbf{C}$ are regular, i.e.\ they preserve regular epimorphisms. See \cite{prolimits} for more details on pro-completions.

Given a map $\varphi\colon c\to c'$ in $\mathbf{C}$, we will write $\Img\varphi\in\mathbf{C}$ for the image of $\varphi$. Recall, from the definition of regular categories, that every map $\varphi\colon c\to c'$ can be (uniquely) factorised as a regular epimorphism $c\sur \Img\varphi$ followed by a monomorphism $\Img\varphi\inj c'$. If $b\inj c$ is a subobject, we might write $\varphi(b)$ for the image $\Img(b\inj c\to c')$.

\begin{lemma}\label{subinv}
If $c=\varprojlim c_i\in\mathbf{C}$, where $c_i\in\mathbf{C}$, with projection maps $\varphi_i\colon c\to c_i$ and $b\inj c$ is a subobject, then $b=\varprojlim \varphi_i(b)$. In particular, every object of $\mathbf{C}$ can be expressed as an inverse limit with epic projection (and hence transition) maps.
\begin{proof}
The composition $b\to\varprojlim\varphi_i(b)\to\varprojlim c_i=c$ is monic, but the first map is also regular epic since inverse limits are regular.
\end{proof}
\end{lemma}

We will also need the following lemma about profinite spaces.

\begin{lemma}\label{stable}
Let $X=\varprojlim_{i\in I} X_i$ be an inverse limit of finite sets, say with projection maps $f_i\colon X\to X_i$ and transition maps $f_{ij}\colon X_i\to X_{j}$. Then, for every $m\in I$, there exists $k\geq m$ such that $$\Img f_{km}=\Img f_m\subseteq X_m.$$
\begin{proof}
Note that the subsets $\{\Img f_{im}\}_{i\geq m}$ of $X_m$ form a directed family, i.e.\ if $i\geq j\geq m$, then $\Img f_{im}\subseteq\Img f_{jm}$. Since $X_m$ is finite, their intersection must be equal to one of them, say $\Img f_{km}$. It only remains to show that $$\Img f_m=\bigcap_{i\geq m}\Img f_{im}.$$ The left-to-right inclusion is clear. Conversely, suppose $x_m\in\bigcap_{i\geq m}\Img f_{im}$. For each $i\geq m$, we have $f_{im}^{-1}(x_m)\neq\varnothing,$ and a standard compactness argument shows that $\varprojlim_{i\geq m}f_{im}^{-1}(x_m)\neq\varnothing.$ But an element $x\in\varprojlim_{i\geq m}f_{im}^{-1}(x_m)\subseteq X$ satisfies $f_m(x)=x_m$.
\end{proof}
\end{lemma}

We will now assume in addition that $\mathbf{D}$ is closed under subobjects, i.e.\ if $d\in\mathbf{D}$ and $c\inj d$ is a subobject in $\mathbf{C}$, then $c\in\mathbf{D}$. We will also assume that coproducts in $\mathbf{C}$ have monic coprojection maps, i.e.\ every canonical map $c\to c\coprod c'$ is monic. This means in particular that if $A\in\mathbf{CoSh}(X,\mathbf{C})$ and $U\subseteq V$ are clopen, then $A(U)\inj A(V)$. Let us give some examples to show that these conditions are not too restrictive.

\begin{exmp}\label{regeg}
\begin{enumerate}[label=(\roman*)]
\item\label{regeg1} The category of models of a finitary algebraic theory over a regular base category is itself regular, by \cite[Theorem 5.11]{barr}. Thus, since $\Set$ is regular, so are the categories $\Grp$ of groups, $\Ab$ of abelian groups, $\Set(G)$ of $G$-sets (for a fixed group $G$) and $\ModR$ of $R$-modules (for a fixed ring $R$). Alternatively, $\Ab$ and $\ModR$ are regular because they are abelian.
\item\label{regeg2} A full subcategory of a regular category which is closed under finite limits and coequalisers is itself regular. Thus, the category $\Set_\mathrm{fin}$ of finite sets is regular, and so are the categories $\Grp_\mathrm{fin}$, $\Ab_\mathrm{fin}$, $\Set(G)_\mathrm{fin}^\mathrm{abs}$ and $\ModR_\mathrm{fin}^\mathrm{abs}$ by the property in \ref{regeg1}. The abbreviation ``abs" stands for ``abstract" and is a reminder that no topology is involved so far.
\item\label{regeg3} Let $G$ be a profinite group and consider the category $\Pro(G)$ of profinite $G$-spaces. We note that $\Pro(G)=\Pro(\Set(G)_\mathrm{fin})$, where $\Set(G)_\mathrm{fin}$ is the category of finite $G$-sets, but where the action of $G$ is still continuous. There is no obvious way to express $\Set(G)_\mathrm{fin}$ as models of a finitary algebraic theory, since we need to remember the topology of $G$. Nevertheless, we observe that $\Set(G)_\mathrm{fin}$ is a full subcategory of $\Set(G)_\mathrm{fin}^\mathrm{abs}$ which is closed under finite limits and coequalisers, so it is in fact regular. Similarly, for a profinite ring $R$, the category $\ModR_{\mathrm{fin}}$ of finite (discrete topological) $R$-modules is regular.
\item\label{regeg4} From the examples above, we see that the following small regular categories are all candidates for $\mathbf{D}$: $\Set_\mathrm{fin}$, $\Grp_\mathrm{fin}$, $\Ab_\mathrm{fin}$, $\Set(G)_\mathrm{fin}$ (for a fixed profinite group $G$) and $\ModR_{\mathrm{fin}}$ (for a fixed profinite ring $R$). The corresponding pro-completion $\mathbf{C}$ is the category $\Pro$ of profinite spaces, $\PGrp$ of profinite groups, $\PAb$ of profinite abelian groups, $\Pro(G)$ of profinite $G$-spaces and $\PModR$ of profinite $R$-modules, respectively. In all cases, $\mathbf{D}$ is closed under subobjects (since it contains literally the finite objects of $\mathbf{C}$), and coproducts in $\mathbf{C}$ have monic coprojections.
\end{enumerate}
\end{exmp}



The following is analogous to \cite[Theorem 5.3.4]{ribesgraph} and \cite[Proposition 2.11]{jc3}.

\begin{lemma}\label{invdecomp}
Every cosheaf $(A,X)\in\mathbf{CoSh}(\mathbf{C})$ can be written as an inverse limit of cosheaves in $\mathbf{CoSh}(\mathbf{C})_\mathrm{fin}$.
\begin{proof}
First, write $X=\varprojlim_{i\in I} X_i$ ($X_i$ finite), with projection maps $f_i\colon X\to X_i$ and transition maps $f_{ii'}\colon X_i\to X_{i'}$, and also write $$A(X)=\varprojlim_{j\in J} d_j\,\,\,\, (d_j\in\mathbf{D}),$$ with projection maps $\varphi_j\colon A(X)\to d_j$ and transition maps $\varphi_{jj'}\colon d_j\to d_{j'}$. Consider the directed poset $I\times J$. For $(i,j)\in I\times J$, define $X_{i,j}=X_i$ and $A_{i,j}(x_i)=\varphi_j(A(f_i^{-1}(x_i)))\inj d_j$, where $x_i\in X_i$. It is then clear that the $\{(A_{i,j},X_{i,j})\}$ form an inverse system in $\mathbf{CoSh}(\mathbf{C})_\mathrm{fin}$ over $I\times J$ and that $X=\varprojlim_{I\times J}X_{i,j}$. Observe that for $i\geq i'$ and $j\geq j'$, the transition map $A_{i,j}\to A_{i',j'}$ is just the map induced by $$\varphi_{jj'}\colon\varphi_j(A(f_{i'}^{-1}(x_{i'})))\to \varphi_{j'}(A(f_{i'}^{-1}(x_{i'}))).$$

We claim that $(A,X)=\varprojlim_{I\times J} (A_{i,j},X_i)$. Consider the map $A\to \varprojlim_{I\times J} A_{i,j}$ in $\mathbf{CoSh}(\mathbf{C})$ (which in fact lives in the non-full subcategory $\mathbf{CoSh}(X,\mathbf{C})$) induced by $\varphi_j\colon A(f_i^{-1}(x_i))\to\varphi_j(A(f_i^{-1}(x_i)))$, where $x_i\in X_i$. In view of Lemma \ref{invsection}, it suffices to show that for each $U\subseteq X$ clopen, the component $A(U)\to\varprojlim_{I\times J}A_{i,j}(f_i(U))$ is an isomorphism.

Pick $i'\in I$ such that for each $i\geq i'$, we have $U=f_{i}^{-1}f_{i}(U)$ (cf.\ the proof of Lemma \ref{invsection}). Now, consider the cofinal subset $S=\{i\in I\colon i \geq i'\}\times J\subseteq I\times J$. By Lemma \ref{subinv}, we have
\begin{eqnarray*}
\varprojlim_{S}A_{i,j}(f_i(U))&=&\varprojlim_{i\geq i'}\varprojlim_j\coprod_{x_i\in f_i(U)}\varphi_j(A(f_i^{-1}(x_i)))\\
&=&\varprojlim_{i\geq i'}\coprod_{x_i\in f_i(U)}\varprojlim_j\varphi_j(A(f_i^{-1}(x_i)))\\
&=&\varprojlim_{i\geq i'}\coprod_{x_i\in f_i(U)}A(f_i^{-1}(x_i))\\
&=&\varprojlim_{i\geq i'}A(U)\\
&=&A(U),
\end{eqnarray*}
which completes the proof.
\end{proof}
\end{lemma}

The following is analogous to \cite[Lemma 2.12]{jc3}.

\begin{lemma}\label{factorone}
Every map $(\varphi,g)\colon\varprojlim_{i\in I}(A_i,X_i)\to (B,Y)$ in $\mathbf{CoSh}(\mathbf{C})$, where the $(A_i,X_i)$ and $(B,Y)$ are in $\mathbf{CoSh}(\mathbf{C})_\mathrm{fin}$, factors through a component $(A_j,X_j)$.
\begin{proof}
Let $(A,X)=\varprojlim(A_i,X_i)$ and let $f_i\colon X\to X_i$ be the projection maps. We will first prove the lemma in the case where the projections $f_i$ are surjective. By Lemma \ref{invsection}, for each $y\in Y$, we have $A(g^{-1}(y))=\varprojlim A_i(f_i(g^{-1}(y)))$. As $B(y)\in\mathbf{D}$, the map $$A(g^{-1}(y))=\varprojlim A_i(f_i(g^{-1}(y)))\to B(y)$$ in $\mathbf{C}$ factors through a component $\psi_y\colon A_j(f_j(g^{-1}(y)))\to B(y)$. Since $Y$ is finite, we can pick a single $j$ that works for every $y\in Y$. Moreover, we can assume that we have picked $j$ such that the map $g\colon X\to Y$ factors through $h\colon X_j\to Y$. Observe that $f_j(g^{-1}(y))= h^{-1}(y)$ for each $y\in Y$, since $f_j$ is surjective. Define a map $\psi\colon A_j\circ h^{-1}\to B$ in $\mathbf{CoSh}(Y,\mathbf{C})$ as follows. For $y\in Y$, the component $A_j(h^{-1}(y))\to B(y)$ is simply $\psi_y$. It is then obvious that $(\psi,h)\colon(A_j,X_j)\to (B,Y)$ is the factorisation we want.

Next, we deal with the general case, where the $f_i$ might not be surjective. We first claim that $(A,X)=\varprojlim (A_i,\Img f_i)$. Indeed, the composition $$(A,X)\to\varprojlim (A_i,\Img f_i)\to\varprojlim(A_i,X_i)=(A,X)$$ is the identity, so the second map is split epic, but it is also easily verified to be monic in $\mathbf{CoSh}(\mathbf{C})$. By what we have proven above, the map $(\varphi,g)\colon(A,X)=\varprojlim (A_i,\Img f_i)\to (B,Y)$ factors through a component $(A_j,\Img f_j)$. It only remains to note that the map $(A,X)\to(A_j,\Img f_j)$ factors through some $(A_k,X_k)$ with $k\geq j$, using Lemma \ref{stable}.
\end{proof}
\end{lemma}

The following is analogous to \cite[Corollary 2.13]{jc3}.

\begin{thm}\label{maineq}
The canonical map $\Pro(\mathbf{CoSh}(\mathbf{C})_\mathrm{fin})\to\mathbf{CoSh}(\mathbf{C})$ is an equivalence.
\begin{proof}
By Proposition \ref{coshRlim} and Lemma \ref{invdecomp}, there is a functor $\Pro(\mathbf{CoSh}(\mathbf{C})_\mathrm{fin})\to\mathbf{CoSh}(\mathbf{C})$ which is essentially surjective. Let $(A,X)=\varprojlim (A_i,X_i)\in\mathbf{CoSh}(\mathbf{C})$ and $(B,Y)\in\mathbf{CoSh}(\mathbf{C})_\mathrm{fin}$, where $(A_i,X_i)\in\mathbf{CoSh}(\mathbf{C})_\mathrm{fin}$. The canonical map $$\tau\colon\varinjlim\Hom_{\mathbf{CoSh}(\mathbf{C})}((A_i,X_i),(B,Y))\to\Hom_{\mathbf{CoSh}(\mathbf{C})}(\varprojlim (A_i,X_i),(B,Y))$$ is surjective by Lemma \ref{factorone}, so it only remains to show that it is also injective.

Let $\{A_i,X_i\}$ have projection maps $(\varphi_i,f_i)$ and transition maps $(\varphi_{ij},f_{ij})$. Suppose, for some fixed $i,j$, that $(\alpha,g)\colon(A_i,X_i)\to(B,Y)$ and $(\beta,h)\colon(A_j,X_j)\to(B,Y)$ have the same image under $\tau$, i.e.\ $(\alpha, g)\circ(\varphi_i,f_i)=(\beta, h)\circ(\varphi_j,f_j)$. We want to show that there is some $k\geq i,j$ such that $$(\alpha,g)\circ(\varphi_{ki},f_{ki})=(\beta,h)\circ(\varphi_{kj},f_{kj}).$$ Since $X=\varprojlim X_l$, we can certainly find a $k$ such that $g\circ f_{ki}=h\circ f_{kj}$, so we can assume that we started with $i=j$ and $g=h$. For each $y\in Y$, let $U_y=f_i^{-1}(g^{-1}(y))\subseteq X$. Using Lemma \ref{invsection}, we obtain two possible compositions $$A(U_y)=\varprojlim_{l\geq i}A_l(f_l(U_y))\to A_i(f_i(U_y))\to A_i(g^{-1}(y))\overset{\alpha}{\underset{\beta}{\rightrightarrows}} B(y)$$ which coincide by assumption. As $B(y)\in\mathbf{D}$, this implies that there is some $m\geq i$ such that the two compositions $$A_m(f_m(U_y))\to A_i(f_i(U_y))\to A_i(g^{-1}(y))\overset{\alpha}{\underset{\beta}{\rightrightarrows}} B(y)$$ already coincide. Since $Y$ is finite, we can pick a single $m$ that works for every $y\in Y$. We have $f_m(U_y)\subseteq f_{mi}^{-1}(g^{-1}(y))$, possibly without equality, but by Lemma \ref{stable}, we can find some $k\geq m$ such that $\Img f_{km}=\Img f_m$. It is then clear that the index $k$ has the property we want.
\end{proof}
\end{thm}

The above theorem makes it easy to verify exactness properties of $\mathbf{CoSh}(\mathbf{C})$ which are closed under pro-completions (see \cite{prostability} for many examples of such properties). As an illustration, we have the following.

\begin{corr}\label{exactcorr}
The category $\mathbf{CoSh}(\mathbf{C})$ is extensive.
\begin{proof}
We note the equivalence $\mathbf{CoSh}(\mathbf{C})_\mathrm{fin}=\mathbf{Fam}(\mathbf{D})$, where the latter is the category of finite families in $\mathbf{D}$, or equivalently the free finite coproduct completion of $\mathbf{D}$. The statement then follows from \cite[Proposition 2.4]{extensive} and the fact that extensivity is closed under pro-completions (see \cite{prostability}).
\end{proof}
\end{corr}

The following is analogous to \cite[Proposition 5.1.7]{ribesgraph}.

\begin{corr}\label{glil}
The global cosections functor $\mathbf{CoSh}(\mathbf{C})\to\mathbf{C}\colon(A,X)\mapsto A(X)$ commutes with inverse limits. That is, if $(A,X)=\varprojlim(A_i,X_i)$ in $\mathbf{CoSh}(\mathbf{C})$, then $A(X)=\varprojlim A_i(X_i)$. In particular, the global cosections functor $\mathbf{CoSh}(\mathbf{C})=\Pro(\mathbf{CoSh}(\mathbf{C})_\mathrm{fin})\to\mathbf{C}=\Pro(\mathbf{D})$ is the (unique) inverse-limit-preserving extension of the functor $$\mathbf{CoSh}(\mathbf{C})_\mathrm{fin}\to\mathbf{C}\colon (A,X)\mapsto A(X)=\coprod_{x\in X}A_x.$$
\begin{proof}
For $d\in\mathbf{D}$, we have natural isomorphisms
\begin{eqnarray*}
\Hom_{\mathbf{C}}(\varprojlim A_i(X_i),d)&=&\varinjlim\Hom_{\mathbf{C}}(A_i(X_i),d)\\
&=&\varinjlim\Hom_{\mathbf{CoSh}(\mathbf{C})}((A_i,X_i),(d,*))\\
&=&\Hom_{\mathbf{CoSh}(\mathbf{C})}(\varprojlim(A_i,X_i),(d,*))\\
&=&\Hom_{\mathbf{CoSh}(\mathbf{C})}((A,X),(d,*))\\
&=&\Hom_{\mathbf{C}}(A(X),d),
\end{eqnarray*}
so we are done by Yoneda.
\end{proof}
\end{corr}

\begin{rmk}\label{corrrmk}\begin{enumerate}[label=(\roman*)]
\item The main conceptual point of Corollary \ref{glil} is that in any pro-regular category $\mathbf{C}=\Pro(\mathbf{D})$ (satisfying our default assumptions), there is a canonical way of defining a ``profinite coproduct". More precisely, we should think of an object $(A,X)\in\mathbf{CoSh}(\mathbf{C})$ as a family of objects $A_x\in\mathbf{C}$ continuously indexed by the profinite space $X$, where $A_x$ is the costalk of $A$ at $x$. The global cosections functor $(A,X)\mapsto A(X)$ is exactly extending the finite coproduct functor $\{A_x\}_{x\in X}\mapsto\coprod_{x\in X}A_x$.
\item\label{usefulrmk1} Recall from usual sheaf theory that there is an equivalence between sheaves and \'etale bundles. Suppose $\mathbf{D}=\Set_\mathrm{fin}$ is the category of finite sets and $\mathbf{C}=\Pro(\mathbf{D})$ is the category of profinite spaces. By Theorem \ref{maineq}, we have the following equivalence between ``profinite cosheaves" and ``profinite bundles" $$\mathbf{CoSh}(\Pro)=\Pro(\mathbf{CoSh}(\Pro)_\mathrm{fin})=\Pro(\Fun(\mathbf{2},\Set_\mathrm{fin}))=\Fun(\mathbf{2},\Pro),$$ where $\mathbf{2}$ is the interval category. This restricts to an equivalence $$\mathbf{CoSh}(X,\Pro)=\Pro_{/X}$$ for every $X\in\Pro$.

Concretely, the functor from left to right sends a cosheaf $A\in\mathbf{CoSh}(X,\Pro)$ to $\bigsqcup_{x\in X} A_x\to X$, where $A_x$ is the costalk of $A$ at $x$ and $\bigsqcup_{x\in X} A_x\subseteq A(X)\times X$ has the subspace topology, which is profinite. The functor from right to left sends a bundle $p\colon E\to X$ to the cosheaf $U\mapsto p^{-1}(U)$.

\item\label{usefulrmk2} Similarly, suppose $\mathbf{D}=\Grp_\mathrm{fin}$ is the category of finite groups and $\mathbf{C}=\Pro(\mathbf{D})$ is the category of profinite groups. By Theorem \ref{maineq} and the group analogue of \cite[Corollary 2.13]{jc3}, we see that the category $\mathbf{CoSh}(\mathbf{PGrp})$ of ``profinite cosheaves of profinite groups" is equivalent to the category of ``profinite bundles of profinite groups", as defined in \cite[Section 5.1]{ribesgraph}. This restricts to an equivalence $$\mathbf{CoSh}(X,\PGrp)=\Grp(\Pro_{/X})$$ for every $X\in\Pro$.

If $R$ is a profinite ring and $\mathbf{D}=\ModR_\mathrm{fin}$ is the category of finite (discrete topological) $R$-modules, then we recover the equivalence of cosheaves and bundles in \cite[Theorem 3.5]{gareth}.

We want to point out that \cite[Theorem 3.2]{jc4} is proven from the perspective of bundles instead of cosheaves. The commutativity of the diagrams in \cite[Theorem 3.2]{jc4} required proof, but it is obvious from the perspective of cosheaves.

\item\label{usefulrmk3} The equivalence $$\mathbf{CoSh}(X,\PGrp)=\Grp(\Pro_{/X})=\Grp(\mathbf{CoSh}(X,\Pro))$$ from combining \ref{usefulrmk1} and \ref{usefulrmk2} is trickier than it first looks. For example, products in $\mathbf{CoSh}(X,\Pro)$ require cosheafification, so for a group object $G\in\Grp(\mathbf{CoSh}(X,\Pro))$, its cosections $G(U)$ do not have to be profinite groups. On the other hand, products in $\Pro_{/X}$ are just pullbacks over $X$ in $\Pro$, so the costalks $G_x$ are profinite groups.

Put differently, the forgetful functor $\mathbf{CoSh}(X,\PGrp)\to\mathbf{CoSh}(X,\Pro)$ is induced by the forgetful functor $\PGrp\to\Pro$ costalkwise, not cosectionwise. As a concrete example, consider the terminal object $T\in\mathbf{CoSh}(X,\PGrp)$, which has trivial cosections and costalks, i.e.\ $T(U)=*$ and $T_x=*$. Forgetting the group structures on costalks, we get the terminal object of $\Pro_{/X}$, i.e.\ the object $X\to X$, which then corresponds to the cosheaf $T'\in\mathbf{CoSh}(X,\Pro)$ defined by $T'(U)=U$. Observe that the cosections object $U$ indeed does not need to be a profinite group.
\end{enumerate}
\end{rmk}

\bibliographystyle{unsrt}

\end{document}